\documentclass[12pt]{amsart}

\newtheorem{theorem}{Theorem}[section]

\newtheorem{corollary}[theorem]{Corollary}

\newtheorem{lemma}[theorem]{Lemma}

\theoremstyle{definition}

\newtheorem{remark}[theorem]{Remark}

\theoremstyle{parrafo}

\numberwithin{equation}{theorem}

\begin{document}

\title[]{Concentration of the ratio between the geometric and
arithmetic means}

\author{J. M. Aldaz}
\address{PERMANENT ADDRESS: Departamento de Matem\'aticas y Computaci\'on,
Universidad  de La Rioja, 26004 Logro\~no, La Rioja, Spain.}
\email{jesus.munarrizaldaz@dmc.unirioja.es}

\address{CURRENT ADDRESS: Departamento de Matem\'aticas,
Universidad  Aut\'onoma de Madrid, Cantoblanco 28049, Madrid, Spain.}
\email{jesus.munarriz@uam.es}


\thanks{The author was partially supported by Grant MTM2006-13000-C03-03 of the
D.G.I. of Spain}







\begin{abstract}  We  explore
the concentration properties of the ratio between the 
 geometric mean and the 
arithmetic mean, showing that for certain sequences of weights one does
obtain concentration, around a value that depends on the sequence. 
\end{abstract}


\maketitle


\markboth{J. M. Aldaz}{Concentration of the geometric mean}

\section
{Introduction.}  This paper is motivated by the article \cite{GluMi}
by
 E. Gluskin and V. Milman, who considered the   variant
$\prod_{i=1}^n |y_i|^{1/n}  
\le 
\sqrt{n^{-1} \sum_{i=1}^n y^2_i}$
 of the AM-GM inequality in the
equal weights case. Roughly speaking, they
showed that the ratio 
$\prod_{i=1}^n |y_i|^{1/n}/
\sqrt{n^{-1} \sum_{i=1}^n y^2_i}$ is bounded below by  $0.394$ asymptotically in $n$
and with high probability, where probability refers to Haar measure on
the euclidean unit sphere $\mathbb{S}_2^{n-1}$. Thus,  the geometric and arithmetic means are comparable
quantities on  ``large sets", provided   $n$ is large.

For the most part, we   focus on the standard 
AM-GM inequality 
$ 
\prod_{i=1}^n |x|_i^{\alpha_{i,n}}  
\le 
 \sum_{i=1}^n \alpha_{i,n} |x|_i,
$
where $\alpha_{i,n} > 0$ and $\sum_{i=1}^n \alpha_{i,n} = 1$.
We shall see that in some cases the concentration
of measure phenomenon does take place:
For certain sequences of weights, which include the equal
weights case, the GM-AM ratio is ``almost constant", with probability
approaching $1$ as $n\to\infty$. Here probability means normalized
area over a suitable unit sphere (a level set of the arithmetic mean 
$f(x) = \sum_{i=1}^n \alpha_{i,n} |x|_i$) which depends on the weights.
We prove that the GM-AM ratio cannot concentrate around values larger
than $e^{-\gamma}\approx 0.5615$, where $\gamma$ is Euler's constant. 
On the other hand, for every $t\in [0, e^{-\gamma}]$ it is possible to find
a family of weights such that concentration takes place around $t$.
In particular, when all weights are equal (to $1/n$ for $n = 2, 3, \dots$)
there is concentration around $e^{-\gamma}$, a result  to
some extent anticipated in \cite[Theorem 5.1]{Gl}. 

For completeness, we also study Gluskin and Milman's 
equal weights modification of the GM-AM ratio: 
$\prod_{i=1}^n |y_i|^{1/n}/
\sqrt{n^{-1} \sum_{i=1}^n y^2_i}$. Here the natural choice of
probability is  the uniform measure $P^{n-1}_2$
on $\mathbb{S}_2^{n-1}:= \{\|y\|_2 = 1\}$, where  
$\|y\|_2:= \sqrt{\sum_{i=1}^n y^2_i}$
is the $\ell_2^n$ norm on $\mathbb{R}^{n}$. 
For $n>>1$ and on 
$\mathbb{S}_2^{n-1}$, the preceding ratio concentrates
around $\sqrt{2} \exp\left[\Gamma^\prime(1/2)/(2\Gamma(1/2))\right]\approx 0.5298$.

The method of proof used here is essentially the same as in \cite{GluMi}:  Compute the
$s$-moment of $\prod_{i=1}^n |x|_i^{\alpha_{i,n}}$ and 
then  use Chebyshev's inequality. The differences lie in the fact 
that we consider the usual AM-GM inequality rather than a variant of
it, we do not restrict ourselves to the equal weights case, and 
 we  optimize over $s$ to obtain concentration results 
 rather than lower
bounds.

The statistical properties of the  ratio between the geometric and the
arithmetic means, in addition to its mathematical interest, may be relevant in those areas where this ratio is used as a measure of homogeneity or information, such as, for instance, radar imaging
(cf. \cite{BMTE}, \cite{Wo}).

\section{Definitions and results.}

Given  $n\ge 2$,  weights
$\alpha_{i,n} > 0$ with $\sum_{i=1}^n \alpha_{i,n} = 1$,
and $x = (x_1,\dots,x_n) \in \mathbb{R}^n$, 
the AM-GM inequality states that
\begin{equation}\label{AMGM} 
\prod_{i=1}^n |x|_i^{\alpha_{i,n}}  
\le 
 \sum_{i=1}^n \alpha_{i,n} |x|_i.
\end{equation}
Assume $x\ne 0$.
Since the middle term of 
\begin{equation}\label{ratio1} 
0\le 
\frac{\prod_{i=1}^n |x_i|^{\alpha_{i,n}}}{ \sum_{i=1}^n \alpha_{i,n}|x_i|}
\le 1
\end{equation}
is a positive homogeneous function of degree zero, it is constant
on the rays $t v$, where $t \in \mathbb{R}$ and $v \in \mathbb{R}^n$.
Thus, we may fix any suitable $t$ and then select 
$v\in \left\{\sum_{i=1}^n \alpha_{i,n}|x_i| = t\right\}$ ``at random", that is, uniformly over this level set. Observe that the
 notion of random choice we use depends on the
sequence $\alpha_{n} := (\alpha_{1,n},\dots,\alpha_{n,n})$. 
Let us renormalize the weights by setting
\begin{equation}\label{ai}
a_{i,n}:= n \alpha_{i,n}.
\end{equation}
Define the unit ball in $\mathbb{R}^n$ associated to 
$a_{n} := (a_{1,n},\dots, a_{n,n})$ by 
$\mathbb{B}^n_{a_n} := \{\sum_{i=1}^n  a_{i,n}|x_i| \le 1\}$,
and denote the corresponding norm and unit sphere by 
$\|\cdot\|_{a_n}$ and 
$\mathbb{S}^{n-1}_{a_n} := \{\sum_{i=1}^n a_{i,n}|x_i| = 1\}$
respectively. In the case of equal weights,
$a_{i,n} = 1$, $\|\cdot\|_{a_n}$ is just the $\ell_1$-norm
$\|\cdot\|_{1}$, 
$\mathbb{B}^n_{a_n}$ the cross-polytope $\mathbb{B}^n_{1}$,
and $\mathbb{S}^{n-1}_{a_n}$ the $\ell_1$-unit sphere $\mathbb{S}^{n-1}_{1}$.
For arbitrary $a_n$, $\mathbb{B}^n_{a_n}$ and $\mathbb{S}^{n-1}_{a_n}$
are simply linear images of $\mathbb{B}^n_{1}$,
and    $\mathbb{S}^{n-1}_{1}$.

To select  $v\in \{\sum_{i=1}^n a_{i,n}|x_i| = t\}$ at random, 
by  homogeneity we can
just set $t=1$, and then
use the normalization of area on $\mathbb{S}^{n-1}_{a_n}$.  
Denote by $\mathcal{H}^s$ the $s$-dimensional
Hausdorff measure on $\mathbb{R}^n$. We adopt the convention that  $\mathcal{H}^s$ is defined via the Euclidean $\ell_2^n$ metric, and normalized by the factor
$\pi^{s/2}/\Gamma(1+s/2)$, so if $s=d$ is a positive natural number, the cube of sidelength
1 has Hausdorff $d$ measure 1. A consistent choice of normalizing factor is required, for 
instance,  
to use Fubini's Theorem, or more generally, the coarea formula.
We denote by   $P_{a_n}^{n-1}$ the uniform probability on 
$\mathbb{S}^{n-1}_{a_n}$, that is 
$P_{a_n}^{n-1} = c \mathcal{H}^{n-1}$, where $c$ is
chosen so that $P_{a_n}^{n-1}(\mathbb{S}^{n-1}_{a_n}) = 1$. Occasionally,
and for simplicity, we will utilize absolute value signs to denote volumes and areas. 

The first three results deal with families of uniformly bounded
renormalized weights. The fourth considers a sequence with some weights
approaching $\infty$, which yields  concentration around zero.

\begin{theorem} \label{looseconc} Suppose that there exist an $M$ 
such that $ a_{i,n} \le M$ for all $n \ge 2$ and all $i= 1,\dots, n$.
Let $\gamma$ be Euler's constant, let $k, \varepsilon > 0$, and let 
$P_{a_n}^{n-1}$ be the uniform
probability on $\mathbb{S}_{a_n}^{n-1}$. Then there exists an 
$N = 
N(k,\varepsilon)$ such that for every $n\ge N$, the ratio 
 of the geometric mean over the arithmetic mean
on $\mathbb{S}_{a_n}^{n-1}$, given by $n\prod_{i=1}^n |x_i|^{\alpha_{i,n}}$,
 satisfies
\begin{equation}\label{concentration1}
P_{a_n}^{n-1}\left\{(1 - \varepsilon) \frac{e^{-\gamma}}{M}
 < n \prod_{i=1}^n |x_i|^{\alpha_{i,n}}
 < (1 + \varepsilon) e^{-\gamma} \right\} \ge 
1 - \frac{1}{n^k}.
\end{equation}
\end{theorem}

In the equal weights case we can set $M=1$, obtaining concentration 
around $e^{-\gamma}\approx 0.5615$.

\begin{corollary} \label{conceqwe} With the same notation as in the preceding
theorem, suppose additionally that $\alpha_{i,n}=1/n$ for $i=1,\dots,n$. Then  there exists an $N = 
N(k,\varepsilon)$ such that for every $n\ge N$, on $\mathbb{S}_{1}^{n-1}$
we have
\begin{equation}\label{concentrationeqwe1}
P_{1}^{n-1}\left\{(1 - \varepsilon)  e^{-\gamma}
 < n \prod_{i=1}^n |x_i|^{1/n}
 < (1 + \varepsilon)  e^{-\gamma} \right\}\ge 
1 - \frac{1}{n^k}.
\end{equation}
\end{corollary}

\begin{remark} The bounds given in Theorem \ref{looseconc} are optimal
in the sense that the upper bound for concentration is achieved in the equal weights case,
and it is also possible to obtain concentration around values arbitrarily
close to  $e^{-\gamma}/M$ for every $M\ge 1$, cf. Remark \ref{remcon} below.
\end{remark}

\begin{remark}\label{rem} When studying the concentration of the GM-AM ratios, in addition to the uniform probabilities
on the spheres $\mathbb{S}_{a_n}^{n-1}$, there are other reasonable  
notions of random choice of a vector. Suppose, for instance, that
we are in the equal weights case $a_{1,n} = 1$, with 
$\mathbb{S}_{a_n}^{n-1} = \mathbb{S}_{1}^{n-1}$. 
Denote this particular 
GM-AM ratio by $r_{1,n}$. We might, for example, select points from the whole space $\mathbb{R}^n$, according 
the exponential
density $2^{-n} e^{-\|x\|_1}$ on $\mathbb{R}^{n}$, or the uniform density over $\mathbb{B}_1^{n} (r)$. Of course, these
 probabilities give
more weight to small vectors than to large ones. But since 
the GM-AM ratio is homogeneous
of degree zero, this makes no difference:  The averages of the 
GM-AM ratio
are equal over all spheres 
$\mathbb{S}_1^{n-1}(r)$ centered at zero with radius $r > 0$, 
so by the coarea formula, the same answer 
 is obtained when using the exponential probability over
$\mathbb{R}^{n}$ or the uniform probability
 over $\mathbb{B}_1^{n} (r)$. 
 
 After having proved  Corollary  \ref{conceqwe}, I came
across   a related, large sample statistics result in the literature, cf. \cite[Theorem 5.1]{Gl},
which in 
the special case $k=1$ states that $r_{1,n} \to e^{-\gamma}$ with
probability 1 as $n\to\infty$, where the independent, identically 
distributed random variables $X_i$
take values in $[0,\infty )$ according to an
exponential distribution. Thus, probability refers to the corresponding
product probability in the countably infinite product of positive semiaxes. 
Informally, this result has the same content as Corollary  \ref{conceqwe}: If one chooses a large quantity of non-negative 
numbers at random, then with very high probability the GM-AM
ratio will be very close to $e^{-\gamma}$. Formally however,
Corollary  \ref{conceqwe} seems to be stronger, since by rewriting it
in terms of the exponential distribution, it yields the above result,
while the other direction does not  follow from the statement of
\cite[Theorem 5.1]{Gl} (a problem is that when defining
a probability on an infinite product many new measure zero sets
may appear, which might have large outer measure from the $n$
dimensional viewpoint). The proof  of \cite[Theorem 5.1]{Gl}
is not provided there, 
 so I do not know whether it
might yield
something similar to Corollary  \ref{conceqwe}.
\end{remark}

Concentration may occur around levels different from $e^{-\gamma}$. In fact,
given any $t\in (0, e^{-\gamma}]$ it is possible to find
a uniformly bounded (by a fixed $M = M(t)$) sequence of weights  such that 
 the ratio 
 of the geometric mean over the arithmetic mean
 concentrates around $t$. 
 Observe in the next result that $M^{-\frac{M -1}{M + 1}}$
 maps $\{1 \le M\}$ 
onto $(0,1]$.

\begin{theorem} \label{concarb} Let $1 \le M$ and let $L>> M$. For $n \ge L$ 
 define
the sequence $\{a_{i,n}: i=1,\dots, n\}$ of weights as follows: 
$a_{j,n}= M$ whenever $ j\le n /(M + 1)$, $a_{j,n}= 1/M$ whenever 
$ j\ge 1 + n /(M + 1)$,
and if $j$ is the least integer strictly larger than $n /(M + 1)$, 
then $a_{j,n} \in [1/M, M]$ is chosen so $\sum_{i=1}^n a_{i,n} = n$. 
Then there exists an $N = 
N(k,\varepsilon)$ such that for every $n\ge N$, 
 the GM-AM ratio 
 on $\mathbb{S}_{a_n}^{n-1}$ 
 satisfies
\begin{equation}\label{concentrationarb}
P_{a_n}^{n-1}\left\{(1 - \varepsilon) \frac{e^{-\gamma}}{M^{\frac{M -1}{M + 1}}}
 < n \prod_{i=1}^n |x_i|^{\alpha_{i,n}}
 < (1 + \varepsilon) \frac{e^{-\gamma}}{M^{\frac{M -1}{M + 1}}}\right\}
 \ge 
1 - \frac{1}{n^k}.
\end{equation}
\end{theorem}

\begin{remark}\label{remcon}  Note that $1/M^{\frac{M -1}{M + 1}}$ is close to $1/M$ provided 
$M$ is large. To show optimality of the lower  bound in Theorem \ref{looseconc} for small values of $M > 1$, we can modify the finite
sequences given in the preceding theorem by using $1/M^j$ as the small weights (with $j >>1$)
 instead of $1/M$,
and a few more large weights $M$, in order to add up to $n$. Then
 the  argument proceeds as in the proof of Theorem \ref{concarb}.
\end{remark}

Suppose next that the largest weights for the
$n$-th averages are given
by a function $f(n) < n$ with $\lim_n f(n) = \infty$. Arguing as in Theorem \ref{concarb}  we obtain 
concentration of the GM-AM ratio around 0. 

\begin{theorem} \label{conc0} Let $f(n) < n$ 
satisfy $\lim_n f(n) = \infty$, and 
 define
$\{a_{i,n}: i=1,\dots, n\}$ by
$a_{j,n}= f(n)$ whenever $ j\le n /(f(n) + 1)$, $a_{j,n}= 1/f(n)$ whenever 
$ j\ge 1 + n /(f(n) + 1)$,
and if $j$ is the least integer strictly larger than $n /(f(n) + 1)$, 
then $a_{j,n} \in [1/f(n), f(n)]$ is chosen so $\sum_{i=1}^n a_{i,n} = n$. 
Then there exists an $N = 
N(k,\varepsilon)$ such that for every $n\ge N$, 
 the GM-AM ratio 
 on $\mathbb{S}_{a_n}^{n-1}$ 
 satisfies
\begin{equation}\label{concentrationarb}
P_{a_n}^{n-1}\left\{ n \prod_{i=1}^n |x_i|^{\alpha_{i,n}}
 < \varepsilon \right\} \ge 
1 - \frac{1}{n^k}.
\end{equation}
\end{theorem}

For completeness, we consider the equal weights variant of the  AM-GM inequality 
\begin{equation}\label{AMGMs} 
\prod_{i=1}^n |y_i|^{1/n}  
\le 
\sqrt{\frac{1}{n} \sum_{i=1}^n y^2_i}
\end{equation}
studied in \cite{GluMi}. The arguments are essentially the
same as in the previous cases, save that now the appropriate notion of random
selection is given by  the uniform probability $P_2^{n-1}$ on the euclidean unit sphere 
$\mathbb{S}_2^{n-1} = \{\|y\|_2 = 1\}$ (alternatively one might
use, for instance, the standard gaussian measure on $\mathbb{R}^n$
to obtain the same conclusion, by the reasons given in 
remark \ref{rem}).  On 
$\mathbb{S}_2^{n-1}$,
$
\prod_{i=1}^n |y_i|^{1/n}  
/
\sqrt{\frac{1}{n} \sum_{i=1}^n y^2_i}
=
\sqrt{n}\prod_{i=1}^n |y_i|^{1/n}.
$

\begin{theorem} \label{concs2} Let $k, \varepsilon > 0$ and let $P_2^{n-1}$ be the uniform
probability on  $\mathbb{S}_2^{n-1}$. Then there exists an $N = 
N(k,\varepsilon)$ such that for every $n\ge N$,  \begin{equation}\label{concentration2}
P_2^{n-1}\left\{(1 - \varepsilon)\sqrt{2} \exp \left(\frac{\Gamma^\prime 
\left(\frac{1}{2}\right)}{2 \Gamma\left(\frac{1}{2}\right)}\right) 
< \sqrt{n} \prod_{i=1}^n |y_i|^{1/n} < (1 + \varepsilon) \sqrt{2} \exp \left(\frac{\Gamma^\prime 
\left(\frac{1}{2}\right)}{2 \Gamma\left(\frac{1}{2}\right)}\right)  \right\}\ge 
1 - \frac{1}{n^k}.
\end{equation}
\end{theorem}

\begin{remark}  $\sqrt{2} \exp \left(\frac{\Gamma^\prime 
\left(\frac{1}{2}\right)}{2 \Gamma\left(\frac{1}{2}\right)}\right) 
\approx   0.5298$.
\end{remark}

\begin{remark} Since 
$ 
\sqrt{\frac{1}{n} \sum_{i=1}^n y^2_i}
\ge 
\frac{1}{n} \sum_{i=1}^n y_i
$ by Jensen's inequality, Theorem \ref{concs2}
shows that with arbitrarily high probability 
\begin{equation}\label{ratio1bb} 
r_{1, n} (x) :=
\frac{\prod_{i=1}^n |x_i|^{1/n}}{ \frac{1}{n}\sum_{i=1}^n |x_i|}
\ge (1 - \varepsilon)\sqrt{2} \exp \left(\frac{\Gamma^\prime 
\left(\frac{1}{2}\right)}{2 \Gamma\left(\frac{1}{2}\right)}\right) 
\end{equation}
on the {\em  euclidean} unit sphere  $\mathbb{S}_2^{n-1}$, provided $n$ is sufficiently large. Analogous remarks can be made
about the behavior of the GM-AM ratio obtained from (\ref{AMGMs})
on $\mathbb{S}_1^{n-1}$, by using Corollary \ref{conceqwe}.
\end{remark}

\begin{remark} Often concentration is a consequence of some type of
uniform Lipschitz behavior (cf., for instance, Levy concentration theorem in
$3\frac{1}{2}. 19$ p. 142 of \cite{Gro}). This is not the case 
here. Consider the equal weights ratio $r_{1,n}$, for instance. If we take as our
sample space  either $\mathbb{R}^{n}$ 
or $\mathbb{B}_1^{n} (r)$, then $r_{1,n}$ is not
even continuous, regardless of how $r_{1,n}(0)$ is defined. On
$\mathbb{S}_1^{n-1}$ the Lipschitz constant of $r_{1,n}$ depends on $n$:
Set  $x=(1/n,1/n, 1/n\dots, 1/n)$ and $y=(0, 2/n,1/n,\dots, 1/n)$.
Then 
$r_{1,n}(x) - r_{1,n}(y) = 1 -0 = 1$, while $\|x - y\|_1 = 2/n$. 
\end{remark}

\section{Lemmas and proofs.}

Let us recall the coarea formula 
(for additional information we refer the reader to \cite{Fe}, pp. 248-250, or \cite{EG}, pp. 117-119):
\begin{equation}\label{coarea}
\int_{\mathbb{R}^{n}} g(x) |Jf(x)| dx =
\int_{\mathbb{R}} \int_{\{f^{-1}(t)\}} g (x)  d{\mathcal{H}^{n-1}(x)} d t.
\end{equation}
Here $f$ is assumed to be Lipschitz, and 
$|Jf(x)| :=\sqrt{\operatorname{det} df(x) df(x)^t}$ denotes the
modulus of the Jacobian.  We also remind the reader about some well known facts 
regarding the $\Gamma$ function (cf. \cite{Lu}, for instance): i) It admits the asymptotic expansion
\begin{equation}\label{asym}
\Gamma (z) = e^{-z} z^{z - 1/2} \sqrt{2\pi} \left(1 + \frac{1}{12 z} 
+ O(z^{-2})\right).
\end{equation}
ii)
$\Gamma(1/2) =\pi^{1/2}$. iii) $\Gamma^\prime (1) = -\gamma$
(Euler's constant). iv) The logarithmic derivative of the
$\Gamma$ function is an analytic function on 
$\mathbb{C}\setminus \{0, -1, -2, \dots\}$; in particular, it is
continuous there.

The following notation, introduced in the preceding section, is recalled
here: $\alpha_{i,n} > 0
$,
$\sum_{i=1}^n\alpha_{i,n} = 1$, and
$a_{i,n} = n \alpha_{i,n}$.
Without loss of generality we assume that the weights are
arranged in decreasing order: $a_{1,n} \ge a_{2,n} \ge \dots \ge a_{n,n}$.
The associated unit sphere and the uniform probability on it are
denoted by $\mathbb{S}_{a_n}^{n-1}$ and $ P_{a_{n}}^{n-1}$ respectively,
and the norm, by $\|\cdot\|_{a_n}$.

\begin{lemma}\label{lemma1} Let $s\in\mathbb{R}$ satisfy $1 + s a_{1,n} >0$. Then the expectation  of $\prod_{i=1}^n |x_i|^{a_{i,n} s}$ over $\mathbb{S}_{a_n}^{n-1}$ is given by
\begin{equation}\label{mean1}
E \left( \prod_{i=1}^n |x_i|^{a_{i,n} s}\right) 
:= 
\int_{\mathbb{S}_{a_n}^{n-1}} \left(  \prod_{i=1}^n |x_i|^{a_{i,n} s}\right) d P_{a_{n}}^{n-1}(x)
= 
\frac{\Gamma\left(n \right)}{\Gamma\left((1 + s) n\right)} \left(  \prod_{i=1}^n \frac{\Gamma\left( 1 + a_{i,n} s \right)}{a_{i,n}^{a_{i,n} s}}\right).
\end{equation}
\end{lemma}

By  renormalization of the weights, $a_{1,n} \ge 1$. 
The restriction $s  > - 1 / a_{1,n}$ is of no consequence to us
since we will be using this result for $s < 0$ (to obtain lower bounds) 
when $\alpha_{1,n} \le M$ and $s\approx 0$
(but $s >>1/n$).

\begin{proof} It is well known and easy to check that
the volume of the  $(\mathbb{R}^n, \|\cdot\|_1)$-unit ball is
$|\mathbb{B}_1^{n}| = 2^{n}/n!$. Let $T_n$ be the linear
transformation satisfying $T_n (\mathbb{S}_{1}^{n-1}) =
\mathbb{S}_{a_n}^{n-1} = \{\sum_{i=1}^n a_{i,n} |x_i| = 1\}$.
Then $\det T_n = \prod_{i=1}^n  a_{i,n}^{-1}$.
Now for every
$x\in \mathbb{R}^{n}\setminus \cup_{i = 1}^n \{x_i\ne 0\}$,
the function
$f(x)= \|x\|_{a_n}$ is differentiable, and $df(x) df(x)^t$ is
a $1\times1$ matrix, so  
$$
|Jf(x)|=\sqrt{\operatorname{det} (df(x) df(x)^t)}= \sqrt{\sum_{i=1}^n a_{i,n}^2} = \|a_n\|_2
$$ a.e. on
$\mathbb{R}^{n}$.  
Set
$g(x) = 1/|Jf(x)|$ in (\ref{coarea}), and denote
by $\mathbb{S}^{n-1}_{a_n}(\rho)$ the sphere centered at 0 of
radius $\rho$ (when $\rho = 1$ we usually omit it). Then 
\begin{equation*}
\frac{ 2^{n}}{n! \prod_{i=1}^n  a_{i,n}} 
= 
|T_n (\mathbb{B}^{n}_{1})| 
= 
|\mathbb{B}^{n}_{a_n}|
= \int_{\mathbb{B}^{n}_{a_n}}   d x
=
\int_{0}^{1} \int_{\mathbb{S}^{n-1}_{a_n}(\rho)} \frac{1}{\|a_n\|_2} d \mathcal{H}^{n-1}(x) d \rho 
\end{equation*}
\begin{equation*}
=
\frac{|\mathbb{S}^{n-1}_{a_n}|}{\|a_n\|_2}\int_{0}^{1} \rho^{n-1} d \rho
 =   
 \frac{|\mathbb{S}_{a_n}^{n-1}|}{ n \|a_n\|_2},
\end{equation*}
so $|\mathbb{S}_{a_n}^{n-1}| = 2^n \|a_n\|_2/(\Gamma(n) \prod_{i=1}^n  a_{i,n})$.

Next, set
$g(x) = \prod_{i=1}^n |x_i|^{a_{i,n} s} \exp\left( - \sum_{i=1}^n a_{i,n} |x_i| \right) /|Jf(x)|$ in (\ref{coarea}). Then
\begin{equation*}
\int_{\mathbb{R}^n} \left( \prod_{i=1}^n |x_i|^{a_{i,n} s}\right) 
\exp\left( - \sum_{i=1}^n a_{i,n} |x_i| \right) dx
\end{equation*}
\begin{equation*}
=
 \frac{1}{\|a_n\|_2} \int_0^\infty e^{- t} \int_{\{\|x\|_{a_n} = t\}} \left( \prod_{i=1}^n |x_i|^{a_{i,n} s}\right)
d\mathcal{H}^{n-1}(x) dt
\end{equation*}
\begin{equation*}
=  \frac{1}{\|a_n\|_2} \int_0^\infty  t^{n s}e^{- t} \int_{\{\|x\|_{a_n} = t\}} \left( \prod_{i=1}^n \left|\frac{x_i}{t}\right|^{a_{i,n} s}\right)
d\mathcal{H}^{n-1}(x) dt
\end{equation*}
\begin{equation*}
 =  \frac{1}{\|a_n\|_2} \int_0^\infty  t^{n s}e^{- t} \int_{\{\|x\|_{a_n} = 1\}} \left( \prod_{i=1}^n \left|x_i\right|^{a_{i,n} s}\right)
t^{n-1} d\mathcal{H}^{n-1}(x) dt
\end{equation*}
\begin{equation*}
 =  \frac{1}{\|a_n\|_2} \int_{\mathbb{S}_{a_n}^{n-1}} \left( \prod_{i=1}^n \left|x_i\right|^{a_{i,n} s}\right) d P_1^{n-1}(x) \frac{2^n \|a_n\|_2}{\Gamma(n) \prod_{i=1}^n  a_{i,n}} \int_0^\infty  t^{(1 + s) n -1 }e^{- t} dt.
 \end{equation*}
Since
\begin{equation*}
\int_{\mathbb{R}^n} \left( \prod_{i=1}^n |x_i|^{a_{i,n} s}\right) 
\exp\left( - \sum_{i=1}^n a_{i,n} |x_i| \right) dx
\end{equation*}
\begin{equation*}
= 
2^n 
\prod_{i=1}^n \left(\int_0^\infty  t^{{a_{i,n} } s}e^{- {a_{i,n}} t} dt\right)
=
 2^n \prod_{i=1}^n \frac{\Gamma(1 + a_{i,n} s)}{a_{i,n}^{a_{i,n} s -1}}
\end{equation*}
and 
\begin{equation*}
\int_0^\infty  t^{(1 + s) n -1 }e^{- t} dt = \Gamma\left((1 + s) n\right),
\end{equation*}
(\ref{mean1}) follows.
\end{proof}

\begin{lemma}\label{lemmagamma} Fix $M \ge 1$ and $\varepsilon \in (0, 1)$. Then there exists a $\delta > 0$  such that for all $s\in(-\delta, \delta)\setminus \{0\}$
and every $t\in (0, M]$, 
we have
\begin{equation}\label{egamma}
(1 - \varepsilon)   e^{-t\gamma} < \Gamma(1+ st)^{1/s} < 
(1 + \varepsilon)  e^{-t\gamma}.
\end{equation}
\end{lemma}

\begin{proof} For $s\ne 0$ sufficiently close to $0$, 
the term $\Gamma\left(1 + s\right)^{1/s}$ can be estimated using L'H\^opital's rule
and the continuity of $\Gamma^\prime (x)/\Gamma (x)$ at 1:
\begin{equation}\label{factor2}
\lim_{s\to 0} \Gamma\left(1 + s\right)^{1/s}
= \exp\left(\frac{\Gamma^\prime (1)}{\Gamma (1)}\right) = e^{-\gamma}.
\end{equation}
Thus, we can choose $\delta >0$ such that for all $s \in (-\delta M, \delta M)\setminus \{0\}$,
\begin{equation}\label{egamma1}
(1 - \varepsilon)^{1/M} e^{-\gamma} <  \Gamma(1+ s)^{1/s}  < 
(1 + \varepsilon)^{1/M} e^{-\gamma}.
\end{equation}
Next, let $s \in (-\delta, \delta)\setminus \{0\}$ and pick $t\in (0, M]$.
Using the change of variables $s\mapsto st$ in (\ref{egamma1}), we get
(\ref{egamma}).
\end{proof}

\begin{lemma}\label{lemmaprod} Suppose there exists an $M \ge  1$ such that
for all $n\in \mathbb{N}$, $n \ge 2$, and all $i=1, \dots, n$, 
$a_{1,n}\le M$. Then for every $\varepsilon > 0$, there exists a $\delta > 0$  such that for every $s\in(-\delta, \delta)\setminus \{0\}$ and
every $n\in \mathbb{N}$ with $n \ge 2$ 
we have
\begin{equation}\label{prod}
(1 - \varepsilon)   e^{-\gamma} < 
\left(  \prod_{i=1}^n \Gamma\left( 1 + a_{i,n} s \right)\right)^{1/sn} < 
(1 + \varepsilon)  e^{-\gamma}.
\end{equation}
\end{lemma}

\begin{proof} Fix $n$. Set $t = a_{i,n}$ in (\ref{egamma}), to get, for 
each $i=1,\dots, n$, 
\begin{equation*}
(1 - \varepsilon)  e^{-a_{i,n}\gamma} <  \Gamma(1+ s a_{i,n})^{1/s}  < 
(1 + \varepsilon) e^{-a_{i,n} \gamma}.
\end{equation*}
Taking the product over $i$ and using $\sum_{i=1}^n a_{i,n} = n$ we obtain 
(\ref{prod}).
\end{proof}

\begin{lemma}\label{lemmalagran} For $i=1, \dots, n$ let $t_i > 0$.
Subject to the restriction $\sum_{i=1}^n t_{i} \ge  n$, the function
$f(t_1,\dots, t_n):= \prod_{i=1}^n t_{i}^{t_i}$ achieves its global
minimum when $t_1 = t_2 = \dots = t_n = 1$, where it takes the value $1$.
\end{lemma}

\begin{proof} Setting $t_i^{t_i} =1$ (by continuity) when $t_i=0$,
the   function
$f(t_1,\dots, t_n):= \prod_{i=1}^n t_{i}^{t_i}$ extends to the set where
$t_i \ge 0$ for  $i=1, \dots, n$. Using Lagrange multipliers
when all $t_i > 0$ and induction when one or several of the $t_i$ equal 0,
the result follows.
\end{proof}

{\em Proof of Theorem \ref{looseconc}.} Fix $k > 0$ and 
$\varepsilon\in (0,1)$ 
($k$ denotes a large parameter and $\varepsilon$ a small one). 
To obtain an upper bound for 
$r_{1,n} (x) = n \prod_{i=1}^n |x_i|^{\alpha_{i,n}}$ on $\mathbb{S}^{n-1}_{a_n}$,
we use Chebyshev's inequality
\begin{equation}\label{cheby1}
P_{a_n}^{n-1}\left\{\prod_{i=1}^n |x_i|^{a_{i,n} s} \ge t\right\} \le 
\frac{1}{t} E\left(\prod_{i=1}^n |x_i|^{a_{i,n} s}\right)
\end{equation} 
with $s = s(k, \varepsilon) > 0$ selected very close to 0, but
fixed (so eventually $s >> n^{-1}$). 
Set
\begin{equation}\label{tchosen}
t = 2 n^k E\left(\prod_{i=1}^n |x_i|^{a_{i,n} s}\right).
\end{equation}
It  follows from (\ref{cheby1})  that 
\begin{equation}\label{concentr1}
P_{a_n}^{n-1}\left\{\prod_{i=1}^n |x_i|^{\alpha_{i,n}} <  t^{1/sn} \right\} \ge 
1 - \frac{1}{2n^k},
\end{equation}
so it suffices to prove 
\begin{equation}\label{t1}
n t^{1/s_0 n} <  (1 + \varepsilon) e^{-\gamma}
\end{equation} 
 for some  small $s_0$ (to be chosen below) and all $n$ large enough; 
in particular, we always assume that $s <<1$, and that $n>> s^{-1}$ whenever both $n$ and
 $s$ appear in the same expression.
Pick $\delta = \delta(\varepsilon) > 0$ such that
$(1 + \delta)^3 < 1 + \varepsilon$. 
By the choice of $t$ (in (\ref{tchosen})) and Lemma \ref{lemma1}, we have
\begin{equation}\label{level1} 
t^{1/sn} = (2 n^k)^{1/sn}   
\left(\frac{\Gamma\left(n \right)}{\Gamma\left((1 + s) n\right)}\right)^{1/sn}
 \left(  \prod_{i=1}^n \frac{\Gamma\left( 1 + a_{i,n} s \right)}{a_{i,n}^{a_{i,n} s}}\right)^{1/sn}. 
\end{equation}
Next we bound each of the three factors in the right hand side of 
(\ref{level1}). From Lemmas \ref{lemmaprod} and \ref{lemmalagran}
we get, for all $s > 0$ sufficiently small, 
\begin{equation}\label{factor3}
\left(  \prod_{i=1}^n \frac{\Gamma\left( 1 + a_{i,n} s \right)}{a_{i,n}^{a_{i,n} s}}\right)^{1/sn}
\le (1 + \delta) e^{-\gamma}. 
\end{equation}
 From the asymptotic expansion (\ref{asym}) of $\Gamma$
 and the assumption $n>> s^{-1}$ 
we obtain 
\begin{equation}\label{factor1}
\left(\frac{\Gamma\left(n \right)}{\Gamma\left((1 + s) n\right)}\right)^{1/sn} 
\le \left(\frac{e}{n}\right) \left(\frac{(1 + s)^{1/2s n}}
{(1 + s)^{1 + 1/s}}\right) \left(1 + O\left(\frac1n\right)\right)^{1/2s n}
\le \left(\frac{1}{n}\right) \left(\frac{e}
{(1 + s)^{1/s}}\right). 
\end{equation}
Finally, by L'H\^opital's rule, given any $s>0$ we have 
\begin{equation}\label{factor3}
\lim_{n\to\infty} (2 n^k)^{1/sn} = 1.
\end{equation}
Thus, we can select $s_0 = s_0(\varepsilon) >0$ so small in (\ref{factor1})
  that $e/(1 + s_0)^{1/s_0} < 1 + \delta$.
Choosing $N = N(k, s_0(\varepsilon))$ such that for all
$n\ge N$ we have $(2 n^k)^{1/s_0 n} < 1+\delta$, inequality (\ref{t1}) follows.

Observe that the hypothesis $a_{i,n} \le M$ on the renormalized weights entails that their geometric
mean $\prod_{i=1}^n a_{i,n}^{\alpha_{i,n}}$ is also bounded above by $M$.
To obtain
\begin{equation*}
P_1^{n-1}\left\{(1 - \varepsilon)\frac{ e^{-\gamma}}{M} < n \prod_{i=1}^n |x_i|^{\alpha_{i,n}} \right\} \ge 
1 - \frac{1}{2n^k},
\end{equation*}
basically all we need to do is to follow the same steps as before, 
but using $s < 0$ instead of $s > 0$, and the preceding observation
instead of Lemma \ref{lemmalagran}. So we  avoid the repetition.
\qed 

\vskip .2 cm

The preceding proof works by respectively giving upper and lower bounds
for $n t^{1/s n}$ when $s > 0$ and $s < 0$ are close enough to zero. Since
\begin{equation*}\label{level1} 
n t^{1/sn} = (2 n^k)^{1/sn}   
\left[n \left(\frac{\Gamma\left(n \right)}{\Gamma\left((1 + s) n\right)}\right)^{1/sn}\right]
 \left(  \prod_{i=1}^n \frac{\Gamma\left( 1 + a_{i,n} s \right)}{a_{i,n}^{a_{i,n} s}}\right)^{1/sn} 
\end{equation*}
and the first two factors on the right hand side approach $1$ as $n\to\infty$,   concentration is controlled by the third factor.
Furthermore, by Lemma \ref{lemmaprod}, for all $s$
sufficiently close to zero we have  
\begin{equation}
(1 - \varepsilon)   e^{-\gamma} < 
\left(  \prod_{i=1}^n \Gamma\left( 1 + a_{i,n} s \right)\right)^{1/sn} < 
(1 + \varepsilon)  e^{-\gamma}, 
\end{equation}
so in order to determine how 
$$
\left(  \prod_{i=1}^n \frac{\Gamma\left( 1 + a_{i,n} s \right)}{a_{i,n}^{a_{i,n} s}}\right)^{1/sn}
$$ 
behaves it is enough to estimate its denominator. The 
concentration results in Theorems \ref{concarb} and \ref{conc0} are obtained by giving  sequences of
weights for which the behavior of 
$
 \prod_{i=1}^n a_{i,n}^{a_{i,n}/n}
$ is easily
understood.

\vskip .2 cm

{\em Proof of Theorem \ref{concarb}.} Denote by $j$ the integer part
of $n/(M + 1)$. Suppose first that $n/(M + 1)$ is an integer.
Then the sequence of weights $a_{i,n} = M$ for $i\le j$ and
$a_{i,n} = 1/M$ for $i > j$ satisfies $\sum_{i=0}^n a_{i,n}=n$. If
$n/(M + 1)$ is not an integer, then $j M + (n-j)/M < n$ while
$(j + 1) M + (n-j -1)/M > n$. Thus, there exists a $t \in[1/M,M]$
such that redefining $a_{j + 1, n} = t$ (instead of $1/M$) we  
have $\sum_{i=0}^n a_{i,n}=n$.  Since for this family of weights 
$\lim_{n\to\infty} \prod_{i=1}^n a_{i,n}^{a_{i,n}/n} = M^{\frac{M-1}{M+1}}$,
the result follows by using the same argument as in the proof of 
Theorem \ref{looseconc}. \qed

\vskip .2 cm

Theorem \ref{conc0} is proven in the same way as the previous one, and
Theorem \ref{concs2}, as Theorem \ref{looseconc}, save for the fact that
we use the uniform probability $P^{n-1}_2$ on the euclidean unit sphere, and the 
corresponding 
 expectation  of $\prod_{i=1}^n |y_i|^{s}$ over $\mathbb{S}_2^{n-1}$.
 This expectation is
computed within the proof of Proposition 1 in \cite{GluMi}:
\begin{equation}\label{expectation2}
E \left( \prod_{i=1}^n |y_i|^{s}\right) := 
\int_{\mathbb{S}_2^{n-1}} \left( \prod_{i=1}^n |y_i|^{s}\right) d P^{n-1}_2
= \left(\frac{\Gamma\left(\frac{1 + s}{2}\right)}{\Gamma\left(\frac{1}{2}\right)}\right)^n \frac{\Gamma\left(\frac{n}{2}\right)}{\Gamma\left(\frac{1 + s}{2} n\right)}.
\end{equation}


\begin{thebibliography}{WWW}


\bibitem[BMTE]{BMTE}  Beauchemin, M.;   Thomson, K.P.B.; Edwards, G. {\em The ratio of the arithmetic to the geometric mean: a
 first-order statistical test for multilook SAR 
image homogeneity.}  IEEE Transactions on
Geoscience and Remote Sensing, 1996, vol. 34, 604--606. 

\bibitem[EG]{EG} Evans, Lawrence C.; Gariepy, Ronald F.
{\em Measure theory and fine properties of functions.}
Studies in Advanced Mathematics. CRC Press, Boca Raton, FL, 1992.

\bibitem[Fe]{Fe} Federer, Herbert {\em Geometric measure theory.}
Die Grundlehren der mathematischen Wissenschaften,
Band 153 Springer-Verlag New York Inc., New York 1969.

\bibitem[Gl]{Gl} Glaser, Ronald E. {\em The ratio of the geometric mean to the arithmetic mean 
for a random sample from a gamma distribution.} J. Amer. Statist. Assoc. 71 (1976), 
no. 354, 480--487. 



\bibitem[GluMi]{GluMi} Gluskin, E.; Milman, V. {\em Note on the geometric-arithmetic mean inequality.} Geometric aspects of functional analysis, 130--135, Lecture Notes in Math., 1807, Springer, Berlin, 2003. 

\bibitem[Gro]{Gro} Gromov, M. {\em Metric Structures for Riemannian
and Non-Riemannian Space.} Progress in Mathematics, Birkha\"auser Boston, 2001. 

\bibitem[Lu]{Lu} Luke, Yudell L. {\em The Special Functions and Their
Approximations.} Volume I, Academic Press, 1969. 

\bibitem[Wo]{Wo} Woodhouse Iain H. {\em The ratio of the arithmetic to the geometric mean: A cross-entropy interpretation.}
IEEE transactions on geoscience and remote sensing,  2001, vol. 39, no. 1, 188--189.


\end{thebibliography}
\end{document}